\documentclass[12pt]{article}
\usepackage{xcolor, graphicx}
\usepackage{amsmath}
\usepackage{amsfonts}
\usepackage{latexsym, amssymb}
\usepackage{pgfplots}
\usepackage{url}

\newtheorem{theorem}{Theorem}[section]

\newtheorem{conjecture}[theorem]{Conjecture}

\newtheorem{definition}[theorem]{Definition}
\newtheorem{example}[theorem]{Example}

\newtheorem{proposition}[theorem]{Proposition}

\newenvironment{proof}[1][Proof]{\textbf{#1.} }{\ \rule{0.5em}{0.5em}}

\def\N{\mathbb{N}}

\def\Z{\mathbb{Z}}
\def\O{\mathcal{O}}

\newcommand{\hs}{\hspace{0.06cm}}

\DeclareMathOperator{\jp}{jp}
\DeclareMathOperator{\ft}{ft}
\DeclareMathOperator{\sjp}{sjp}

\DeclareMathOperator{\syr}{syr}
\DeclareMathOperator{\sft}{sft}
\DeclareMathOperator{\odd}{\mathbb{O}}

\date{}

\begin{document}
\title{Is the Syracuse falling time bounded by 12?}

\author{Shalom Eliahou\footnote{eliahou@univ-littoral.fr}, Jean Fromentin\footnote{fromentin@univ-littoral.fr}\hspace{0.01cm} and R\'enald Simonetto\footnote{renalds@microsoft.com}}

\maketitle

\begin{abstract}
Let $T \colon \mathbb{N}\to \mathbb{N}$ denote the $3x+1$ function, where $T(n)=n/2$ if $n$ is even, $T(n)=(3n+1)/2$ if $n$ is odd. As an accelerated version of $T$, we define a \emph{jump} at $n \ge 1$ by $\textrm{jp}(n) = T^{(\ell)}(n)$, where $\ell$ is the number of digits of $n$ in base 2. We present computational and heuristic evidence leading to surprising conjectures. The boldest one, inspired by the study of $2^{\ell}-1$ for $\ell \le 500\hs000$, states that for any $n \ge 2^{500}$, at most four jumps starting from $n$ are needed to fall below $n$, a strong form of the Collatz conjecture.
\end{abstract}

\begin{quote}
\end{quote}

\section{Introduction}

We denote by $\N$ the set of positive integers. Let $T \colon \N \to \N$ be the notorious $3x+1$ function, defined by $T(n)=n/2$ if $n$ is even, $T(n)=(3n+1)/2$ if $n$ is odd. For $k \ge 0$, denote by $T^{(k)}$ the $k$th iterate of $T$. The \emph{orbit} of $n$ under $T$ is the sequence
$$
\O_T(n)=(n,T(n),T^{(2)}(n),\dots).
$$

The famous Collatz conjecture states that for all $n \ge 1$, there exists $r \ge 1$ such that $T^{(r)}(n)=1$. The least such $r$ is denoted $\sigma_\infty(n)$ and called the \emph{total stopping time} of $n$. An equivalent version of the Collatz conjecture states that for all $n \ge 2$, there exists  $s \ge 1$ such that $T^{(s)}(n)<n$. The least such $s$ is denoted by $\sigma(n)$ and called the \emph{stopping time} of $n$. For instance, we have
\begin{equation}\label{sigma le 2}
\sigma(n)=
\begin{cases}
1 &\textrm{ if $n$ is even,}\\
2 &\textrm{ if $n \equiv 1 \bmod 4$,}
\end{cases}
\end{equation}
as is well known and easy to check. A \emph{stopping time record} is an integer $n \ge 2$ such that $\sigma(m) < \sigma(n)$ for all $2 \le m \le n-1$.

For the original slower version $C \colon \N \to \N$, where $C(n)=n/2$ or $3n+1$ according as $n$ is even or odd, the analog of the stopping time is called the \emph{glide} in~\cite{R}. The list of all currently known \emph{glide records}, complete up to at least $2^{60}$, is maintained in~\cite{R2}. It is quite likely that glide records and stopping time records coincide; we have verified it by computer up to $2^{32}$.

\smallskip
It is well known that $\sigma(n)$ is unbounded as $n$ grows. For instance, since
\begin{equation}\label{2^l-1}
T^{(\ell)}(2^\ell-1)=3^\ell-1,
\end{equation}
as follows from the formula $T(2^a3^b-1)=2^{a-1}3^{b+1}-1$ for $a \ge 1$, we have $\sigma(2^\ell-1) \ge \ell$ for all $\ell \ge 2$.

In this paper, we propose an accelerated version of the function $T$. The idea, somewhat as in~\cite{T}, is to apply an iterate of $T$ to $n$ depending on the number of digits of $n$ in base 2. Accordingly, we introduce the following function.

\begin{definition} The \emph{jump function} $\jp \colon \N \to \N$ is defined for $n \in \N$ by
$$
\jp(n) = T^{(\ell)}(n),
$$
where $\ell=\lfloor \log_2(n) +1\rfloor$ is the number of digits of $n$ in base~2.
\end{definition}

\begin{example} \emph{We have $\jp(1) = T^{(1)}(1)=2$, and $\jp(2)=T^{(2)}(2)=2$ since $2$ is of length $\ell=2$ in base $2$. For $n=27$, written $11011$ in base 2, hence of length $\ell=5$, we have $\jp(27)=T^{(5)}(27)=71$. In turn, $71$ is of length $\ell=7$ in base $2$ since $2^6 \le 71 < 2^7$, whence $\jp(71)=T^{(7)}(71)=137$. The orbit of $27$ under jumps is displayed below in~\eqref{jumps from 27}.}
\end{example}

\begin{example} \emph{A single jump at $n=2^\ell-1$ with $\ell \ge 1$ yields
\begin{equation}
\jp(2^\ell-1)=3^\ell-1.
\end{equation}
This follows from the equalities $\ell=\lfloor \log_2(2^\ell-1) +1\rfloor$ and~\eqref{2^l-1}.}
\end{example}

\begin{example}\label{jp(2n)} \emph{We have $\jp(2n)=\jp(n)$ for all $n \ge 1$. Indeed, $2n$ is of length one more than $n$ in base $2$.}
\end{example}

In analogy with the stopping time relative to $T$, we now introduce the falling time relative to jumps. As $\jp(1)=\jp(2)=2$, we only consider $n \ge 3$.

\begin{definition} Let $n \ge 3$. The \emph{falling time} of $n$, denoted $\ft(n)$, is the least $k \ge 1$ such that $\jp^{(k)}(n) < n$, or $\infty$ if there is no such $k$.
\end{definition}

Note that, for a presumed cyclic orbit under $T$ with minimum $m \ge 3$, we would have $\jp(m)=\infty$.

There is no tight comparison between stopping time and falling time. It may happen that $\sigma(a) < \sigma(b)$ whereas $\ft(a) > \ft(b)$. For instance, for $a=41$ and $b=43$, we have
\begin{align*}
\sigma(41)=2 & \,<\, \sigma(43)=5, \\
\ft(41)=8 & \,>\, \ft(43)=2.
\end{align*}
It may also happen that $\ft(n) > \sigma(n)$, as shown by the case $n=41$.

\medskip
Of course, the Collatz conjecture is equivalent to $\ft(n) < \infty$ for all $n \ge 3$.
In Section~\ref{sec ft}, we provide computational evidence leading us to a stronger conjecture, namely that $\ft(n)$ is in fact bounded for all $n \ge 3$. Specifically, all integers $n$ we have tested so far satisfy $\ft(n) \le 16$. See Conjecture~\ref{conjecture ft}. In Section~\ref{sec syracuse}, in analogy with the falling time, we introduce the \emph{Syracuse falling time} $\sft(n)$, and corresponding conjectures, by only considering the odd terms in the orbits $\O_T(n)$. In Section~\ref{sec 2^l-1}, we report surprising computational results on $\ft(2^\ell-1)$ and $\sft(2^\ell-1)$ for $\ell \le 500\hs000$, and we formulate corresponding conjectures. In the last Section~\ref{sec large}, inspired by the case $n=2^{\ell}-1$, we formulate still stronger conjectures on $\ft(n)$ and $\sft(n)$ for very large integers $n$.  We conclude the paper with some supporting heuristics.

For a wealth of information, developments and commented references related to the $3x+1$ problem, see the webpage and book of J. C. Lagarias~\cite{L1,L2}.
To date, the Collatz conjecture has been verified by computer up to $2^{68}$ by D. Barina~\cite{B}. Using this bound, it follows from~\cite{E} that any non-trivial cycle of $T$ must have length at least 114\hs208\hs327\hs604.

\section{Falling time records}\label{sec ft}

In this section, we only consider those positive integers $n$ satisfying $\sigma(n) \ge 3$, i.e. such that $n \equiv 3 \bmod 4$ by~\eqref{sigma le 2}. Let us denote by $4\N+3$ the set of those integers. Here is our first computational evidence that the falling time remains small.

\begin{proposition}\label{prop ft 14} We have $\ft(n) \le 14$ for all $n \in [1, 2^{44}-1]$ such that $n \equiv 3 \bmod 4$.
\end{proposition}
\begin{proof} In a few days of computing time with CALCULCO~\cite{Cal}.
\end{proof}

\medskip
As shown in Table~\ref{falling time}, the smallest $n \in 4\N+3$ such that $\ft(n) \ge 14$, namely $n=12\hs235\hs060\hs455$, actually satisfies $\ft(n) = 14$ and $n > 2^{33}$.

\begin{definition}
A \emph{falling time record} is an integer $n \in 4\N+3$ such that $\ft(m) < \ft(n)$ for all $m \in 4\N+3$ with $m < n$.
\end{definition}

\begin{table}[ht]
\caption{Falling time records up to $2^{35}$}
\begin{center}
\begin{tabular}{|r|c|c|} \hline
$n \equiv 3 \bmod 4$ & $\lfloor\log_2(n)+1\rfloor$ & $\ft(n)$ \\
\hline \hline
3 & 2 & 2 \\ \hline
7 & 3 & 3 \\ \hline
27 & 5 & 8 \\ \hline
60\hs975 & 16 & 9 \\ \hline
1\hs394\hs431 & 21 & 10 \\ \hline
6\hs649\hs279 & 23 & 11 \\ \hline
63\hs728\hs127 & 26 & 13 \\ \hline
12\hs235\hs060\hs455 & 34 & 14 \\ \hline
\end{tabular}

\label{falling time}
\end{center}
\end{table}

The list of falling time records up to $2^{35}$ is given in Table~\ref{falling time}. It was built while establishing Proposition~\ref{prop ft 14}. For instance, we have $\ft(3)=2$, $\ft(7)=3$ and $\ft(n) \le 3$ for all $3 \le n < 27$ such that $n \equiv 3 \bmod 4$. The value $\ft(27)=8$ follows from the fact that $8$ jumps are needed from $27$ to fall below it, as shown by the orbit of 27 under jumps:
\begin{equation}\label{jumps from 27}
\O_{\jp}(27) = (27,71,137,395,566,3\hs644,650,53,8, 2,2,\dots).
\end{equation}
Interestingly, five of the falling time records in~Table~\ref{falling time} are also glide records, namely 3, 7, 27,  63\hs728\hs127 and 12\hs235\hs060\hs455, as seen by consulting~\cite{R2}.

Table~\ref{falling time} shows that the number 12 and a few smaller ones fail to occur as falling time records. One may then wonder about the smallest $n \in 4\N+3$ reaching $\ft(n)=12$.

The answer is to be found in Table~\ref{new falling}. Let us define a \emph{new falling time} as an integer $n \in 4\N+3$ such that  $\ft(n)$ is distinct from $\ft(m)$ for all smaller $m \in 4\N+3$. Of course, every falling time record is a new falling time. The list of new falling times we know so far, which are not already falling time records, is given in Table~\ref{new falling}.

\medskip
\begin{table}[ht]
\caption{Some new falling times}
\begin{center}
\begin{tabular}{|c||c|c|c|c|c|}
\hline
$n$ & 111 & 103 & 71 & 55 & 217\hs740\hs015 \\ \hline
$\ft(n)$ & 4 & 5 & 6 & 7 & 12 \\ \hline
\end{tabular}

\label{new falling}
\end{center}
\end{table}

\subsection{Integers satisfying $\ft(n) > 14$}
For $a,b \in \Z$, we denote by $[a,b]=\{n \in \Z \mid a \le n \le b\}$ the integer interval they span. Recalling Example~\ref{jp(2n)}, namely that $\ft(2m)=\ft(m)$ for all $m \ge 3$, a single integer $n$ satisfying $\ft(n) > 14$ yields \emph{infinitely many} integers $N$ satisfying $\ft(N) > 14$, namely $N=2^r n$ for all $r \ge 1$. However, the latter numbers have stopping time equal to $1$, and hence are not particularly interesting.

Only those integers $n$ satisfying $\ft(n) > 14$ and having a reasonably large stopping time are really interesting, for their apparent rarity and their relevance to the Collatz conjecture. Hence, we shall restrict our search to \emph{$24$-persistent} integers, i.e. to those $n$ satisfying
$$\sigma(n) \ge 24.$$
The property for $n$ of being $24$-persistent only depends on its class mod $2^{24}$. See~\cite{Te} for more details on the description of the condition $\sigma(n) \ge k$ by classes mod $2^k$. See also~\cite{Ev}. For $k=24$, the number of $24$-persistent classes mod $2^{24}$ is exactly $286\hs581$.

As shown below, the occurrence of $\ft(n) > 14$ among $24$-persistent numbers seems to be very rare. Here is our first computational result.
\begin{proposition} The smallest $24$-persistent integer $n$ such that $\ft(n) > 14$ is
$$
n_0=1\hs008\hs932\hs249\hs296\hs231.
$$
It satisfies $\ft(n_0)=15$, $\sigma(n_0)=886$ and $\lfloor \log_2(n_0)+1\rfloor = 50$.
\end{proposition}
\begin{proof} In a few days of computing time with CALCULCO~\cite{Cal}.
\end{proof}

\subsubsection{The neighborhood of $g_{30}$}
It turns out that $n_0$ is a glide record, and as such is listed under the name
$$n_0=g_{30}$$
in Table~\ref{records} of Section~\ref{roosendaal}. We have verified by computer that \emph{$g_{30}$ is the smallest $24$-persistent integer $n$ satisfying $\ft(n) > 14$.} However, because of the restriction $\sigma(n) \ge 24$, we do not know whether $g_{30}$ is an actual falling time record.

In a small neighborhood of $g_{30}$ in the Collatz tree, we found $11$ more $24$-persistent integers $n$ satisfying $\ft(n) > 14$. By \emph{small neighborhood} of $g_{30}$, we mean here integers $m$ such that
$$
T^{(i)}(m) = T^{(j)}(g_{30})
$$
for some $i \in [0,40],j \in [0,30]$. It turns out that these $11$ integers all satisfy $\ft(n) = 15$, as $g_{30}$ itself. They are displayed in Table~\ref{24-persistent}.
\begin{table}[ht]
\caption{$24$-persistent integers satisfying $\ft(n)=15$.}
\vspace{-1em}
\begin{align*}
& 1\hs513\hs398\hs373\hs944\hs347, 1\hs702\hs573\hs170\hs687\hs391, 2\hs017\hs864\hs498\hs592\hs463, \\
& 2\hs553\hs859\hs756\hs031\hs087,  3\hs405\hs146\hs341\hs374\hs783, 3\hs830\hs789\hs634\hs046\hs631, \\
& 5107719512062175, 5\hs746\hs184\hs451\hs069\hs947, 6\hs464\hs457\hs507\hs453\hs691, \\ & 7\hs272\hs514\hs695\hs885\hs403, 22\hs370\hs169\hs558\hs105\hs279.
\end{align*}
\vspace{-1em}
\label{24-persistent}
\end{table}

\subsubsection{The neighborhood of $g_{32}$}
There is another glide record in Table~\ref{records} with falling time $15$, namely
$$
g_{32}= 180\hs352\hs746\hs940\hs718\hs527.
$$
In this case, looking at a somewhat larger neighborhood of $g_{32}$ in the Collatz tree, namely at all $m$ such that
$$
T^{(i)}(m) = T^{(j)}(g_{32})
$$
for some $i \in [0,50],j \in [0,30]$, we found four $24$-persistent integers $n$ reaching $\ft(n)=16$. These four integers are displayed in Table~\ref{ft(n)=16}.
\begin{table}[ht]
\caption{$24$-persistent integers satisfying $\ft(n)=16$.}
\vspace{-1em}
\begin{align*}
& 49\hs312\hs2600\hs554\hs790\hs303, \,\,\, 739\hs683\hs900\hs832\hs185\hs455, \\
& 986\hs245\hs201\hs109\hs580\hs607, \,\,\, 1\hs479\hs367\hs801\hs664\hs370\hs911.
\end{align*}
\vspace{-1em}
\label{ft(n)=16}
\end{table}

However, these four integers $n$ have a small stopping time. They all satisfy $\sigma(n) \in [35,48]$, as compared to $\sigma(g_{32})=966$. Hence again, they are not particularly interesting.

\medskip

This leads us to the following conjecture.

\begin{conjecture}\label{conjecture ft} There exists $B \ge 16$ such that $\ft(n) \le B$
for all $n \ge 3$.
\end{conjecture}

An even bolder conjecture, based on the data we currently have, would be to take $B=16$ above. Anyway, with whatever value of $B$, Conjecture~\ref{conjecture ft} constitutes a strong form of the Collatz conjecture.

\subsection{Glide records}\label{roosendaal}

Eric Roosendaal maintains the list of all currently known glide records~\cite{R2}, complete up to at least $2^{60}$. At the time of writing, there are 34 of them, denoted $g_1, \dots, g_{34}$ below. As noted in~\cite{R2}, only the first $32$ ones have been independently checked. The ten biggest are displayed in descending order in Table~\ref{records}. It turns out that
$$
\ft(g_1), \dots, \ft(g_{34}) \le 15.
$$
Moreover, among them, the highest value $\ft(n)=15$ is only reached by $g_{30}$ and $g_{32}$. Table~\ref{ft(n)=16} displays four $24$-persistent integers $n$ satisfying $\ft(n) = 16$ in the neighborhood of $g_{32}$. We do not know whether $\ft(n) \ge 17$ is at all reachable.
\begin{table}[ht]
\caption{Top ten known glide records}
\begin{center}
\begin{tabular}{|c||r|c|c|r|c|} \hline
$n$ &  & $\lfloor \log_2(n)+1\rfloor$ & glide of $n$ & $\sigma(n)$ & $\ft(n)$ \\
\hline \hline
$g_{34}$ & 2\hs602\hs714\hs556\hs700\hs227\hs743 & 62 & 1\hs639 & 1005 & 13 \\ \hline
$g_{33}$ & 1\hs236\hs472\hs189\hs813\hs512\hs351 & 61 & 1\hs614 & 990 & 14 \\ \hline
$g_{32}$ & 180\hs352\hs746\hs940\hs718\hs527 & 58 & 1\hs575 & 966 & \textbf{15} \\ \hline
$g_{31}$ & 118\hs303\hs688\hs851\hs791\hs519 & 57 & 1\hs471 & 902 & 12 \\ \hline
$g_{30}$ & 1\hs008\hs932\hs249\hs296\hs231 & 50 & 1\hs445 & 886 & \textbf{15} \\ \hline
$g_{29}$ & 739\hs448\hs869\hs367\hs967 & 50 & 1\hs187 & 728 & 12 \\ \hline
$g_{28}$ & 70\hs665\hs924\hs117\hs439 & 47 & 1\hs177 & 722 & 13 \\ \hline
$g_{27}$ & 31\hs835\hs572\hs457\hs967 & 45 & 1\hs161 & 712 & 13 \\ \hline
$g_{26}$ & 13\hs179\hs928\hs405\hs231 & 44 & 1\hs122 & 688 & 14 \\ \hline
$g_{25}$ & 2\hs081\hs751\hs768\hs559 & 41 & \hspace{0.28cm}988 & 606 & 12 \\ \hline
\end{tabular}
\label{records}
\end{center}
\end{table}

\subsection{Falling time distribution}

In three distinct graphics, we display the distribution of the values taken by the falling time function in large integer intervals. These graphics show that the proportion of the case $\ft(n) \ge 3$ in the integer intervals $[2^{\ell}, 2^{\ell+1}-1]$ tends to $0$ as $\ell$ grows.

\begin{itemize}
    \item  Figure~\ref{F:Histo} displays the proportion of the occurrence of $\ft(n)=1$, $\ft(n)=2$ and $\ft(n) \ge 3$, respectively, among all odd integers in the integer intervals $[2^{\ell}, 2^{\ell+1}-1]$ for $2 \le \ell \le 40$.

\item  Figure~\ref{F2:Histo} does the same but separates the cases $n \equiv 1 \bmod 4$ and $n \equiv 3 \bmod 4$. The purpose is to show that the former case, with stopping time $2$, behaves like the more interesting latter case, and so may be safely ignored.

\item  Finally, Figure~\ref{F3:Histo} is restricted to $24$-persistent integers in the integer intervals $[2^{\ell}, 2^{\ell+1}-1]$ for $24 \le \ell \le 50$.
\end{itemize}

\pgfplotsset{scale=0.5,ymax=1.1,ytick={0,0.5,1},axis x line=bottom,axis y line=left,ymajorgrids,xlabel={\vspace{-2em}$\ell$},ytick={0,1},xlabel near ticks}

\begin{figure}[ht]
\caption{Proportion of odd integers in $[2^\ell,2^{\ell+1}-1]$ with falling time equal to $1$, $2$ and greater than $2$, respectively. The integer $\ell$ goes from $2$ to $40$}
\vspace{0.5em}
\footnotesize
\begin{center}
\begin{tikzpicture}
\begin{axis}[xmin=1,xtick={2,20,40}]
\addplot[line width=1] table [x=l, y=a, col sep=comma] {datas/histo_40.csv};
\end{axis}
\draw(1.8,2.9) node{$\ft=1$};
\begin{scope}[shift={(4.5,0)}]
\begin{axis}[xmin=1,xtick={2,20,40}]
\addplot[line width=1] table [x=l, y=b, col sep=comma] {datas/histo_40.csv};
\end{axis}
\draw(1.8,2.9) node{$\ft=2$};
\end{scope}
\begin{scope}[shift={(9,0)}]
\begin{axis}[xmin=1,xtick={2,20,40}]
\addplot[line width=1] table [x=l, y=c, col sep=comma] {datas/histo_40.csv};
\end{axis}
\draw(1.8,2.9) node{$\ft\geq3$};
\end{scope}
\end{tikzpicture}
\end{center}
\label{F:Histo}
\end{figure}
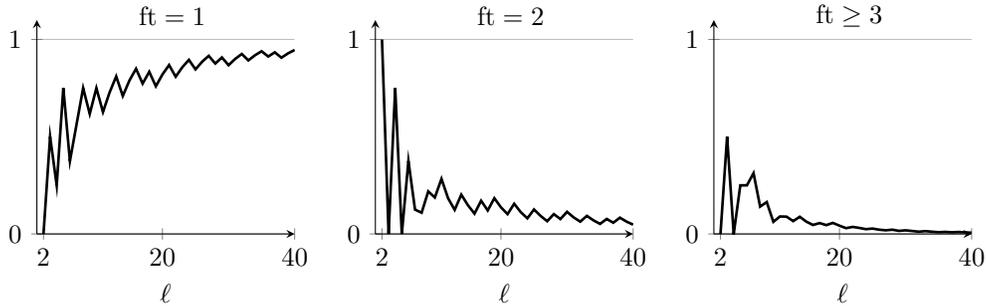

\begin{figure}[ht]
\caption{Same plot as for Figure \ref{F:Histo} except that integers are separated with respect to their class $1$ or $3$ modulo $4$. Gray curves are for integers congruent to $1$ modulo $4$ while black ones are for those congruent to $3$ modulo $4$.}
\vspace{0.5em}
\footnotesize
\begin{center}
\begin{tikzpicture}
\begin{axis}[xmin=1,xtick={2,20,40}]
\addplot[line width=1,color=gray] table [x=l, y=a, col sep=comma] {datas/histo_40_mod_1.csv};
\addplot[line width=1] table [x=l, y=a, col sep=comma] {datas/histo_40_mod_3.csv};
\end{axis}
\draw(1.8,2.9) node{$\ft=1$};
\begin{scope}[shift={(4.5,0)}]
\begin{axis}[xmin=1,xtick={2,20,40}]
\addplot[line width=1,color=gray] table [x=l, y=b, col sep=comma] {datas/histo_40_mod_1.csv};
\addplot[line width=1] table [x=l, y=b, col sep=comma] {datas/histo_40_mod_3.csv};
\end{axis}
\draw(1.8,2.9) node{$\ft=2$};
\end{scope}
\begin{scope}[shift={(9,0)}]
\begin{axis}[xmin=1,xtick={2,20,40}]
\addplot[line width=1,color=gray] table [x=l, y=c, col sep=comma] {datas/histo_40_mod_1.csv};
\addplot[line width=1] table [x=l, y=c, col sep=comma] {datas/histo_40_mod_3.csv};
\end{axis}
\draw(1.8,2.9) node{$\ft\geq3$};
\end{scope}
\end{tikzpicture}
\end{center}
\label{F2:Histo}
\end{figure}
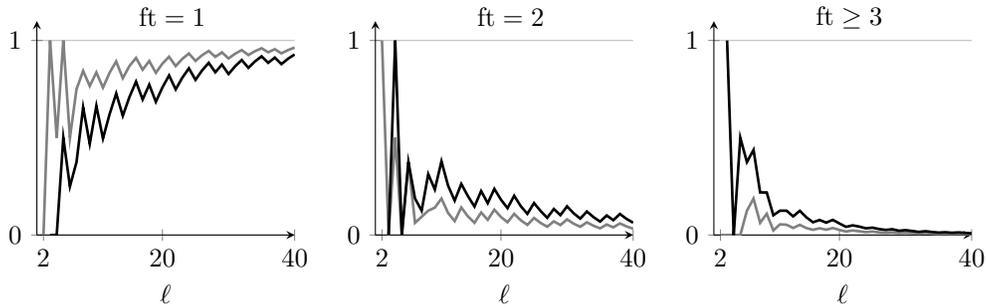

\begin{figure}[ht]
\caption{Proportion of $24$-persistent integers in $[2^\ell,2^{\ell+1}-1]$ with falling time equal to $1$, $2$ and greater than $2$, respectively. The integer $\ell$ goes from~$24$ to~$50$}
\vspace{0.5em}
\footnotesize
\begin{center}
\begin{tikzpicture}
\begin{axis}[xmin=23,xtick={24,30,40,50}]
\addplot[line width=1] table [x=l, y=a, col sep=comma] {datas/histo_24_50_persistent.csv};
\end{axis}
\draw(1.8,2.9) node{$\ft=1$};
\begin{scope}[shift={(4.5,0)}]
\begin{axis}[xmin=23,xtick={24,30,40,50}]
\addplot[line width=1] table [x=l, y=b, col sep=comma] {datas/histo_24_50_persistent.csv};
\end{axis}
\draw(1.8,2.9) node{$\ft=2$};
\end{scope}
\begin{scope}[shift={(9,0)}]
\begin{axis}[xmin=23,xtick={24,30,40,50}]
\addplot[line width=1] table [x=l, y=c, col sep=comma] {datas/histo_24_50_persistent.csv};
\end{axis}
\draw(1.8,2.9) node{$\ft\geq3$};
\end{scope}
\end{tikzpicture}
\end{center}
\label{F3:Histo}
\end{figure}
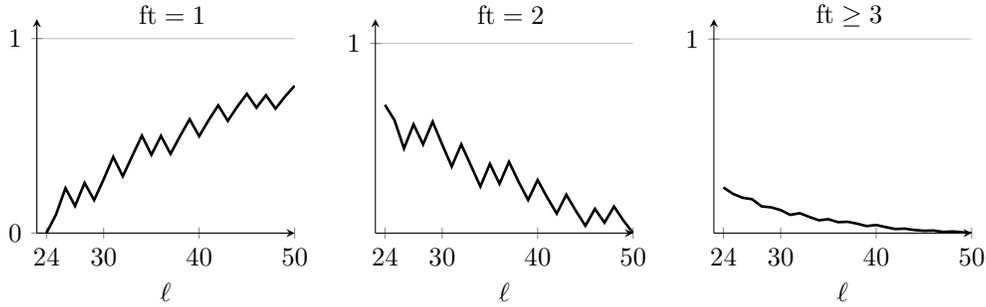

\subsection{A variant of jumps}\label{sec variant}

Let $h \in \N$. For all $n \in \N$, we define
$$
\jp_h(n)=T^{(h \ell)}(n)
$$
where, as before, $\ell$ is the number of digits of $n$ in base 2. This is not the same, of course, as the $h$-iterate $\jp^{(h)}(n)$. Note also that for $h=1$, we recover jumps, i.e. $\jp_1(n)=\jp(n)$. For $n \ge 3$, the \emph{$h$-falling time} $\ft_h(n)$ is defined correspondingly, as the smallest $k \ge 1$, if any, such that $\ft_h^{(k)}(n) < n$.

It turns out that for $h=18$, and for the glide records $g_1,\dots,g_{34}$, we have
$$
\ft_{18}(g_i)=1
$$
for all $1 \le i \le 34$. In view of that fact, is it true that $\ft_{18}(n)=1$ for all $n \ge 3$? We do not know. But we have verified it up to $n \le 2^{30}$, and it cannot be outright dismissed, given the conjectural behavior of $\ft(n)$ for very large $n$ as discussed in Section~\ref{sec large}. Of course, a positive answer would imply the Collatz conjecture. On the other hand, uncovering any counterexample would be quite a feat.

\section{The Syracuse version}\label{sec syracuse}

Let $\odd=2\N+1$ denote the set of odd positive integers. Another well-studied version of the $3x+1$ function is $\syr \colon \odd \to \odd$, defined on any $x \in \odd$ by
$$
\syr(x) = (3x+1)/2^\nu,
$$
where $\nu \ge 1$ is the largest integer such that $2^\nu$ divides $3x+1$. That is, $\syr(x)$ is the largest odd factor of $3x+1$. This specific version is called the \emph{Syracuse function} in~\cite{T}. It has been amply investigated in the past, though under different notation or names. For instance in~\cite{C}, where lower bounds on the length of presumed nontrivial cycles of $\syr(x)$ are given in terms of the convergents $p_n/q_n$ to $\log_2(3)$; or in~\cite{KM,KS,S}, where statistical properties of $\syr(x)$ and related maps are studied using the Structure theorem of Sinai, which we briefly recall in Section~\ref{sec syr variant} below.

\medskip
In analogy with the functions $\jp(n)$ and $\ft(n)$ related to the $3x+1$ function $T(x)$, we now introduce the corresponding functions $\sjp(n)$ and $\sft(n)$ related to the Syracuse version $\syr(x)$.

\begin{definition} We define the \emph{Syracuse jump function} $\sjp \colon \odd \to \odd$ by
$$
\sjp(n) = \syr^{(\ell)}(n), \textrm{ where } \ell=\lfloor \log_2(n) +1\rfloor.
$$
\end{definition}
\begin{example} \emph{We have $\sjp(1)=1$, $\sjp(3)=1$ and $\sjp(27)=\syr^{(5)}(27)=107$.}
\end{example}

Here is the corresponding Syracuse falling time.

\begin{definition} Let $n \in \odd \setminus\{1\}$. The \emph{Syracuse falling time} of $n$, denoted $\sft(n)$, is the least $k \ge 1$ such that $\sjp^{(k)}(n) < n$, or $\infty$ if there is no such $k$.
\end{definition}

\begin{example}\emph{We have $\sft(27)=6$, as witnessed by the orbit of $27$ under Syracuse jumps, namely
\begin{equation*}
\O_{\sjp}(27) = (27,107,233,377,911,53,1,1,\dots).
\end{equation*}}
\end{example}

\smallskip
As one may expect, the inequality $\sft(n) \le \ft(n)$ holds very often, but not always. For instance, for $n=199$, we have $\ft(199)=1$ but $\sft(199)=5$. The former equality follows from the orbit
$$
\O_T(199)=(199,299,449,674,337,506,253,380,190, \dots)
$$
and the value $\lfloor \log_2(199) +1\rfloor=8$, yielding $\jp(199)=190$, while the latter one follows from the orbit
$$
\O_{\syr}(199)=(199, 323, 395, 479, 577, 1, \dots).
$$

\begin{definition}
A \emph{Syracuse falling time record} is an integer $n \in 4\N+3$ such that $n \ge 7$ and $\sft(m) < \sft(n)$ for all $m \in 4\N+3$ with $m < n$.
\end{definition}

\begin{table}[ht]
\caption{Syracuse falling time records up to $2^{35}$}
\begin{center}
\begin{tabular}{|r|c|c|} \hline
$n \equiv 3 \bmod 4$ & $\lfloor \log_2(n)+1\rfloor$ & $\sft(n)$ \\
\hline \hline
7 & 3 & 2 \\ \hline
27 & 5 & 6 \\ \hline
6\hs649\hs279 & 23 & 7 \\ \hline
63\hs728\hs127 & 26 & 9 \\ \hline
\end{tabular}

\label{Syracuse falling}
\end{center}
\end{table}

The complete list of Syracuse falling time records up to $2^{35}$ is displayed in Table~\ref{Syracuse falling}. Compared with Table~\ref{falling time}, it turns out that all current Syracuse falling time records are also falling time records. The converse does not hold, as shown by the falling time records 60\hs975 and 1\hs394\hs431 in Table~\ref{falling time}.

\subsection{Current maximum}

The Collatz conjecture is equivalent to the statement $\sft(n) < \infty$ for all $n \in \odd \setminus \{1\}$. Again, it is likely that a stronger form holds, namely that $\sft(n)$ is \emph{bounded} on $\odd \setminus \{1\}$. Besides the computational evidence above and below, some heuristics point to that possibility in Section~\ref{sec large}. Similarly to Proposition~\ref{prop ft 14}, here is a computational result in that direction.

\begin{proposition}\label{prop sft 10} We have $\sft(n) \le 9$ for all $n \in [3, 2^{35}-1]$ such that $n \equiv 3 \bmod 4$.
\end{proposition}
\begin{proof} By computer with CALCULCO~\cite{Cal}.
\end{proof}

\medskip
As yet another hint pointing to the same direction, it turns out that
\begin{equation}\label{sft for glides}
\sft(g_1), \dots, \sft(g_{34}) \le 10
\end{equation}
for the 34  currently known glide records. For definiteness, Table~\ref{10 largest} displays the Syracuse falling times of the top ten glide records as listed in Table~\ref{records}.

\smallskip
\begin{table}[ht]
\caption{Syracuse falling times of top ten glide records}
\begin{center}
\begin{tabular}{|c||c|c|c|c|c|c|c|c|c|c|}
\hline
$n$ & $g_{25}$ & $g_{26}$ & $g_{27}$ & $g_{28}$ & $g_{29}$ & $g_{30}$ & $g_{31}$ & $g_{32}$ & $g_{33}$ & $g_{34}$ \\ \hline \hline
$\sft(n)$ & 9 & 8 & 8 & 8 & 8 & \textbf{10} & 8 & \textbf{10} & 9 & 8 \\
\hline
\end{tabular}

\label{10 largest}
\end{center}
\end{table}

Among the $g_i$, and as in Section~\ref{roosendaal} for the falling time, only $g_{30}$ and $g_{32}$ reach the current maximum of the Syracuse falling time, namely $\sft(n)=10$. Interestingly, the biggest currently known glide record, namely $n=g_{34}$, only satisfies $\sft(n)=8$. With Proposition~\ref{prop sft 10} and~\eqref{sft for glides} in the background, here is our formal conjecture.

\begin{conjecture}\label{conjecture sft} There exists $C \ge 10$ such that $\sft(n) \le C$ for all $n \equiv 3 \bmod 4$.
\end{conjecture}

Again, the truth of this conjecture would yield a strong positive solution of the Collatz conjecture. At the time of writing, no single positive integer $n \equiv 3 \bmod 4$ is known to satisfy $\sft(n) \ge 11$. Thus, a still bolder conjecture would be to take $C=10$ in Conjecture~\ref{conjecture sft}, or $C=12$ to be on a safer side. Whence the title of this paper.

\subsection{A variant of Syracuse jumps}\label{sec syr variant}

As in Section~\ref{sec variant} for jumps, we propose here an accelerated variant of Syracuse jumps. Let $h \in \N$. For all $n \in \odd$, we define
$$
\sjp_h(n)=\syr^{(h \ell)}(n),
$$
where $\ell$ is the number of digits of $n$ in base 2.  Of course, $\sjp_1(n)=\sjp(n)$. For $n \ge 3$, the \emph{Syracuse $h$-falling time} $\sft_h(n)$ is defined correspondingly, as the smallest $k \ge 1$, if any, such that $\sft_h^{(k)}(n) < n$. It turns out that for $h=12$, and for the glide records $g_1,\dots,g_{34}$, we have
$$
\sft_{12}(g_i)=1
$$
for all $1 \le i \le 34$. Again, we may ask whether $\sft_{12}(n)=1$ holds for all odd $n \ge 3$. A positive answer would imply the Collatz conjecture. We have verified it up to $n \le 2^{30}$, and our semi-random search did not yield any counterexample.

Anyway, the occurrence of $\sft_{12}(n)=1$ as $n$ grows to infinity is probably overwhelming; and, just possibly, tools such as Sinai's structure theorem and its applications~\cite{KM,KS,S} might help prove that this is indeed the case. But we shall not pursue here this line of investigation.

For convenience, let us briefly recall the statement of that theorem. Given $x \in 6\N+\{1,5\}$, let $x_i=\syr^{(i)}(x)$ for all $i \ge 0$, and let $k_i \ge 1$ be such that $x_{i}=(3x_{i-1}+1)/2^{k_i}$ for all $i \ge 1$. Moreover, for $m \ge 1$, set
$$
\gamma_m(x)=(k_1,\dots,k_m).
$$
Sinai's structure theorem states that \emph{given any $(k_1,\dots,k_m) \in \N^m$, the set of all $x \in 6\N+\{1,5\}$ such that $\gamma_m(x)=(k_1,\dots,k_m)$ consists of a unique and full congruence class mod $6\cdot 2^{k_1+\cdots+k_m}$ in $\N$.}

\section{The case $2^\ell-1$}\label{sec 2^l-1}
In sharp contrast with the stopping time of $2^\ell-1$, for which $\sigma(2^\ell-1) \ge \ell$ for all $\ell \ge 2$, the falling time and the Syracuse falling time of $2^\ell-1$ seem to remain very small as $\ell$ grows. Here is some strong computational evidence.

\begin{proposition}\label{prop ft 2^l-1} Besides $\ft(2^5-1)=\ft(2^6-1)=8$, we have $\ft(2^\ell-1) \le 5$ for all $2 \le \ell \le 500\hs000$ with $\ell\notin\{5,6\}$.
\end{proposition}
\begin{proof} In a few days of computing time with CALCULCO~\cite{Cal}.
\end{proof}

\medskip
Moreover, the value $\ft(2^\ell-1)=5$ seems to occur finitely many times only, the last one being presumably at $\ell = 132$. In turn, the value $\ft(2^\ell-1)=4$ seems to occur infinitely often. Whence the following conjecture, verified by computer up to $\ell = 500\hs000$.

\begin{conjecture}\label{conj ft 2^l-1} We have $\ft(2^{\ell}-1) \le 4$ for all $\ell \ge 133$.
\end{conjecture}

\smallskip
Here are the analogous statement and conjecture for the Syracuse falling time.

\begin{proposition}\label{prop sft 2^l-1} Besides $\sft(2^5-1)=\sft(2^6-1)=5$, and $\sft(2^{24}-1) = 4$, we have \vspace{-0.4cm}
\begin{align*}
& \sft(2^\ell-1) \in \{2,3\} \textrm{ for all }\ell \in [2,4\hs624] \setminus \{5,6,24\}, \\
& \sft(2^\ell-1) = 2 \textrm{ for all }\ell \in [4\hs625, 500\hs000].
\end{align*}
\end{proposition}
\begin{proof} In a few days of computing time with CALCULCO~\cite{Cal}.
\end{proof}

\bigskip
This leads us to the following conjecture, true up to $\ell \le 500\hs000$.
\begin{conjecture}\label{conj sft 2^l-1} We have $\sft(2^\ell-1)=2$ for all $\ell \ge 4\hs625$.
\end{conjecture}

\section{For very large $n$}\label{sec large}

As hinted by the computational evidence and conjectures of Section~\ref{sec 2^l-1} on the case $n=2^{\ell}-1$, by intensive semi-random search, and by the heuristics below, it appears to be increasingly difficult for integers $n$ to satisfy $\ft(n) \ge 5$ or $\sft(n) \ge 3$ as they grow very large. Here then are still bolder conjectures.

\begin{conjecture}\label{conj ft}
We have $\ft(n) \le 4$ for all $n \ge 2^{500}$.
\end{conjecture}
This threshold of $2^{500}$ is inspired by Conjecture~\ref{conj ft 2^l-1}, of course with a margin for safety. It cannot be significantly lowered, since $\ft(2^{132}-1)=5$ as noted before Proposition~\ref{prop ft 2^l-1}. Moreover, intensive random search revealed one integer $n \in [2^{70}, 2^{71}-1]$ satisfying $\ft(n)=5$, namely
$$
n = 1\hs884\hs032\hs044\hs420\hs885\hs877\hs201\hs579\hs449\hs071\hs924\hs925\hs072\hs300\hs117\hs065\hs411.
$$
However, this integer is congruent to $3$ mod $16$ and hence has stopping time equal to $4$ only.

\medskip
Here is the analogous conjecture for the Syracuse falling time. Its threshold of $2^{5000}$ is similarly inspired by Proposition~\ref{prop sft 2^l-1} and Conjecture~\ref{conj sft 2^l-1}.

\begin{conjecture}\label{conj sft}
We have $\sft(n) \le 2$ for all odd $n \ge 2^{5000}$.
\end{conjecture}

\subsection{Heuristics}
Besides the computational evidence leading to Conjectures~\ref{conjecture ft}, \ref{conjecture sft}, \ref{conj ft 2^l-1}, \ref{conj sft 2^l-1},~\ref{conj ft}  and~\ref{conj sft}, a heuristic argument would run as follows. It is well known that the Collatz conjecture is equivalent to the statement that, starting with any integer $n \ge 1$, the probability for $T^{(k)}(n)$ to be even or odd tends to $1/2$ as $k$ grows to infinity. Thus, even if $n$ written in base 2 is a highly structured binary string, as e.g. for $n=2^\ell-1$, one may expect that for $\ell=$ the length of that string, then $T^{(\ell)}(n)$ in base 2 will already look more random. That is, a single jump or Syracuse jump at $n \ge 3$ should already introduce a good dosis of randomness, all the more so as $n$ grows very large. And therefore, a bounded number of jumps or Syracuse jumps at $n$ might well suffice to fall below $n$.

\subsection{A challenge}
We hope that the experts in highly efficient computation of the $3x+1$ function will tackle the challenge of probing these conjectures to much higher levels than the ones reported here. For instance, as both a challenge and a request to the reader, and in view of Conjecture~\ref{conj ft}, if you do find any $n \ge 2^{500}$ satisfying $\ft(n) \ge 5$, please e-mail it to the authors. Your solution will be duly recorded on a dedicated webpage.

\section*{Acknowledgements}

The computations performed for this paper were carried out on the CALCULCO high performance computing platform provided by SCoSI/ULCO (Service COmmun du Système d’Information de l’Université du Littoral Côte d’Opale).

\medskip
\noindent
\textbf{Authors' addresses:}

\smallskip

\noindent
Shalom Eliahou\textsuperscript{a,b}, Jean Fromentin\textsuperscript{a,b} and R\'enald Simonetto\textsuperscript{a,b,c}

\noindent
\textsuperscript{a}Univ. Littoral C\^ote d'Opale, UR 2597 - LMPA - Laboratoire de Math\'ematiques Pures et Appliqu\'ees Joseph Liouville, F-62100 Calais, France\\
\textsuperscript{b}CNRS, FR2037, France\\
\textsuperscript{c}Microsoft France, 37 Quai du Pr\'esident Roosevelt, 92130 Issy-les-Moulineaux, France


\begin{thebibliography}{99}
\bibitem{B} D. Barina, \emph{Convergence verification of the Collatz problem}, The Journal of Supercomputing 77 (2021), 2681--2688.
\bibitem{Cal} CALCULCO, a computing platform at Universit\'e du Littoral C\^ote d'Opale.
\bibitem{C} R. E. Crandall, \emph{On the ``\hspace{0.9mm}$3x+1$\hspace{-0.5mm}'' Problem}, Math. Comp. 32 (1978), 1281--1292.
\bibitem{E} S. Eliahou, \emph{The $3x+1$ problem: new lower bounds on nontrivial cycle lengths}, Discrete Math. 11 (1993), 45--56.
\bibitem{Ev} C. J. Everett, \emph{Iteration of the Number-Theoretic Function $f(2n)=n$, $f(2n+1)=3n+2$}, Adv. Math. 25 (1977), 42--45.
\bibitem{KM} A. V. Kontorovich and S. J. Miller, \emph{Benford’s law, values of L-functions, and the $3x + 1$ problem}, Acta Arith. 120 (2005), 269--297.
\bibitem{KS} A. V. Kontorovich and Ya. G. Sinai, \emph{Structure Theorem for $(d, g, h)$-maps}, Bull. Braz. Math. Soc. (N.S.) 33 (2002), 213--224.
\bibitem{L1} J. C. Lagarias, $3x+1$ problem and related problems, \url{https://dept.math.lsa.umich.edu/~lagarias//3x+1.html}
\bibitem{L2} The Ultimate Challenge: The $3x+1$
 Problem. J. C. Lagarias, Editor. Amer. Math. Soc., Providence, RI, 2010.
\bibitem{R} E. Roosendaal, \url{www.ericr.nl/wondrous}
\bibitem{R2} E. Roosendaal, \url{www.ericr.nl/wondrous/glidrecs.html}
\bibitem{S} Ya. G. Sinai, \emph{Statistical $(3x+1)$-Problem}, Comm. Pure Appl. Math. 56 (2003), 1016--1028.
\bibitem{T} T. Tao, \emph{Almost all orbits of the Collatz map attain almost bounded values} (2019) \url{arXiv:1909.03562}
\bibitem{Te} R. Terras, \emph{A stopping time problem on the positive integers}, Acta Arith. 30 (1976), 241--252.
\end{thebibliography}
\end{document}